\documentclass[]{amsart}
\usepackage{epsfig}
\usepackage{amsmath, amssymb}
\usepackage{color}
\numberwithin{equation}{section}

\newtheorem{theorem}{Theorem}[section]
\newtheorem{lemma}[theorem]{Lemma}
\newtheorem{corollary}[theorem]{Corollary}
\newtheorem{proposition}[theorem]{Proposition}

\allowdisplaybreaks \numberwithin{equation}{section}

\newcommand{\weglassen}[1]{}

\renewcommand{\imath}{\mathrm{i}}
\renewcommand{\d}{\mathrm{d}}
\def\ma#1#2#3#4{\left[{}^{#1}_{#3}{}^{#2}_{#4}\right]_2}

\begin{document}
\title[Rosenhain-Thomae formulae]{Rosenhain-Thomae formulae for higher genera hyperelliptic curves}

\author{Keno Eilers}
\address{Faculty of Mathematics, University of Oldenburg,
 Carl-von-Ossietzky-Str. 9-11, 26129 Oldenburg, Germany}
\email{keno.eilers@uni-oldenburg.de}


\begin{abstract} Rosenhain's famous formula expresses the periods of first kind integrals of genus two 
hyperelliptic curves in terms of $\theta$-constants. In this paper we generalize the Rosenhain formula to higher genera hyperelliptic curves by means of the second Thomae formula for derivated non-singular $\theta$-constants. 
\end{abstract}

\maketitle
\section{Introduction}
The developments of the theory of algebraic curves (and related theories) in the XIX-th century led to the idea of describing and classifying objects relevant to algebraic curves and their Jacobians in terms of their modular forms,
the Riemann $\theta$-functions, which depend on the Riemann period matrix $\tau$. In this respect, a lot of work was accomplished for (hyper-)elliptic curves of genus $1$ and $2$.
In this paper we want to generalize the existing results primarily due to Rosenhain and discuss here such representations of periods of higher genera hyperelliptic integrals.\\
The Riemann period matrix $\tau$ is defined as the quotient, 
$\tau = \mathcal{A}^{-1} \mathcal{B}$ of the $\mathcal{A}$- and $\mathcal{B}$- period matrices of holomorphic integrals. Here, the leading question is the inverse problem: Given the Riemann period matrix $\tau$, how can we express the period matrix $\mathcal{A}$ in terms of $\theta$-constants and, possibly, invariants of the curve?\\
The $\theta$-constant representation of a complete elliptic integral, 
$$K=\frac{\pi}{2}\theta_{3}^2(0),$$ was known since Jacobi's times. Rosenhain, Jacobi's student, obtained a generalization of this formula to genus-$2$-curves in terms of $\theta$-constants with characteristics \cite{ros851}. To remind this
result we introduce a genus two curve,
\begin{equation}
y^2 = x(x-1)(x-a_1 )(x- a_2 )(x- a_3 ),\quad a_i\in \mathbb{C}. 
\end{equation}
For a given Riemann matrix $\tau$, we denote as $\mathcal{A}$ and $\mathcal{B} = \mathcal{A}\tau$ the period matrices. Also 
we define $\theta$-constants with even characteristics $[\varepsilon]= 
\left[ {}^{\varepsilon_1'}_{\varepsilon_1}  {}^{\varepsilon_2'}_{\varepsilon_2}  \right]$, 
$\boldsymbol{\varepsilon}'\boldsymbol{\varepsilon}^T=0 \; \mathrm{mod}\; 2$,  
$\varepsilon_i,\varepsilon_j'\in \mathbb{Z}_+$ as the Riemann-$\theta$-functions (with characteristics) evaluated at zero. That is, specifically:
\begin{align}\begin{split}
\theta[\varepsilon] =\sum_{\boldsymbol{n}\in \mathbb{Z}^2}  \mathrm{exp} \;  \imath \pi 
\left\{
\left(n_1+\frac{\varepsilon_1'}{2}  \right)^2\tau_{1,1} +2\left(n_1+\frac{\varepsilon_1'}{2}  \right)
\left(n_2+\frac{\varepsilon_2'}{2}  \right) \tau_{1,2}+\left(n_2+\frac{\varepsilon_2'}{2}  \right)^2\tau_{2,2}\right.\\
 \left.+\varepsilon_1\left(n_1+\frac{\varepsilon_1'}{2}  \right)+\varepsilon_2\left(n_2+\frac{\varepsilon_2'}{2}  \right)  
\right\}\neq 0.\end{split}\label{eventhetaconst}
\end{align}
The derivated odd $\theta$-constants for odd characteristics $[\delta]= 
\left[ {}^{\delta_1'}_{\delta_1}  {}^{\delta_2'}_{\delta_2}  \right]$, 
$\boldsymbol{\delta}'\boldsymbol{\delta}^T=1 \; \mathrm{mod}\; 2$,  
$\delta_i,\delta_j'\in \mathbb{Z}_+$, are defined as the derivation of the Riemann-$\theta$-functions, evaluated at zero:
\begin{align}\begin{split}
\theta_i[\delta] =2\imath\pi \sum_{\boldsymbol{n}\in \mathbb{Z}^2} \left(n_i+ \frac{\delta'_i}{2} \right)  
&\mathrm{exp} \; 
 \imath{\pi} \left\{
\left(n_1+\frac{\delta_1'}{2}  \right)^2\tau_{1,1} +2\left(n_1+\frac{\delta_1'}{2}  \right)
\left(n_2+\frac{\delta_2'}{2} \right) \tau_{1,2}\right. \\
 &\left.+\left(n_2+\frac{\delta_2'}{2}  \right)^2\tau_{2,2}
+\delta_1\left(n_1+\frac{\delta_1'}{2}  
\right)+\delta_2\left(n_2+\frac{\delta_2'}{2}  \right)\right\}\neq 0\end{split}\label{oddthetaconst}, \quad i=1,2
\end{align}
For genus-$2$-curves, there are 16 characteristics. 6 of them are odd and 10 even and we denote the sets of characteristics 
as $\mathcal{S}_6$ and $\mathcal{S}_{10}$, correspondingly. Odd characteristics are in $1-1$ correspondence 
with the branch points, $(0,1,a_1,a_2,a_3,\infty)$, in a way which will become clear in the subsequent sections.\\
\begin{figure}[h]
\begin{center}
\unitlength 0.7mm \linethickness{0.6pt}
\begin{picture}(150.00,80.00)
\put(9.,33.){\line(1,0){12.}} \put(9.,33.){\circle*{1}}
\put(21.,33.){\circle*{1}} \put(10.,29.){\makebox(0,0)[cc]{$e_1$}}
\put(21.,29.){\makebox(0,0)[cc]{$e_2$}}
\put(15.,33.){\oval(20,30.)}
\put(8.,17.){\makebox(0,0)[cc]{$\mathfrak{ a}_1$}}
\put(15.,48.){\vector(1,0){1.0}}
\put(32.,33.){\line(1,0){9.}} \put(32.,33.){\circle*{1}}
\put(41.,33.){\circle*{1}} \put(33.,29.){\makebox(0,0)[cc]{$e_3$}}
\put(42.,29.){\makebox(0,0)[cc]{$e_4$}}
\put(37.,33.){\oval(18.,26.)}
\put(30.,19.){\makebox(0,0)[cc]{$\mathfrak{a}_2$}}
\put(36.,46.){\vector(1,0){1.0}}
\put(100.,33.00) {\line(1,0){33.}} \put(100.,33.){\circle*{1}}
\put(133.,33.){\circle*{1}}
\put(101.,29.){\makebox(0,0)[cc]{$e_{5}$}}
\put(132.,29.){\makebox(0,0)[cc]{$e_{6}=\infty$}}
\put(59.,58.){\makebox(0,0)[cc]{$\mathfrak{b}_1$}}
\put(63.,62.){\vector(1,0){1.0}}
\bezier{484}(15.,33.00)(15.,62.)(65.,62.)
\bezier{816}(65.00,62.)(119.00,62.00)(119.00,33.00)
\bezier{35}(15.,33.00)(15.,5.)(65.,5.)
\bezier{35}(65.00,5.)(119.00,5.00)(119.00,33.00)
\put(70.,44.){\makebox(0,0)[cc]{$\mathfrak{b}_2$}}
\put(74.00,48.){\vector(1,0){1.0}}
\bezier{384}(37.,33.00)(37.,48.)(76.00,48.)
\bezier{516}(76.00,48.)(111.00,48.00)(111.00,33.00)
\bezier{30}(37.,33.00)(37.,19.)(76.00,19.)
\bezier{30}(76.00,19.)(111.00,19.00)(111.00,33.00)
\end{picture}
\end{center}
\caption{Homology basis of the Riemann surface of the curve
$V(x,y)$ with real branching points $e_1 < e_2 <\ldots <
e_{6}=\infty$ (upper sheet).  The cuts are drawn from $e_{2i-1}$
to $e_{2i}$, $i=1,2,3$.  The $\mathfrak b$--cycles are completed
on the lower sheet (dotted lines).} \label{figure-1}
\end{figure}
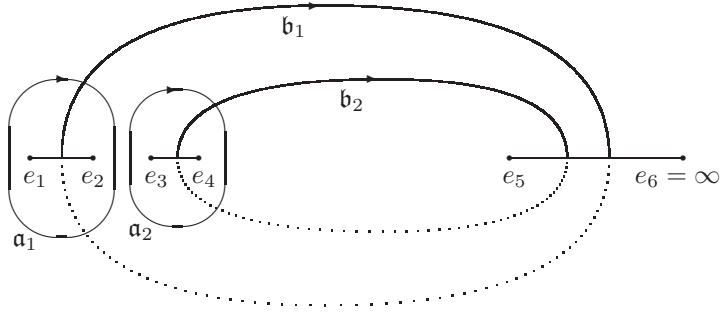

Rosenhain's central theorem is taken from his only known publication\footnote{Wikipedia tells us: "Rosenhain galt auch als begabt in Sprachen und Musik; allerdings bemerkten einige Beobachter, dass er die hohen Erwartungen seiner jungen Jahre nicht erf\"ullte und nach seiner preisgekr\"onten Arbeit keine nennenswerten Beitr\"age mehr ver\"offentlichte."}, where it was indeed unproven. We, too, give it here without proof:
\begin{theorem}[Rosenhain's modular representation of the period matrix $\mathcal{A}$]
In the homology basis given on Fig. \ref{figure-1}, there are characteristics
\begin{align}
\begin{split}
&[\alpha_1] =[_{00}^{11}], \quad  [\alpha_2] =[_{10}^{00}], \quad  [\alpha_3] =[_{01}^{10}]\\
&[\beta_1] =[_{10}^{01}], \quad  [\beta_2] =[_{00}^{10}], \quad  [\beta_3] =[_{11}^{00}]
\end{split}
\end{align}
and 
\begin{equation}
[\delta_1] := [\alpha_1]+[\alpha_2]+[\alpha_3] \mod\; 2 \ = [_{11}^{01}] , \quad  
[\delta_2] := [\beta_1]+[\beta_2]+[\beta_3] \mod\; 2 \ = [_{01}^{11}],
\end{equation} 
such that
\begin{equation}
\mathcal{A}^{-1} = \frac{1}{2\pi^2 Q^2}\left( \begin{array}{rr}  
-P \theta_2[\delta_1]& Q\theta_2[\delta_2]\\
P \theta_1[\delta_1]&-Q \theta_1[\delta_2]
  \end{array} \right), \quad a_1a_2a_3 = \frac{P^4}{Q^4} \label{rosenhain2}
\end{equation}
with the quantities $P$ and $Q$ as abbreviations for:
\begin{equation}
P= \theta[\alpha_1]\theta[\alpha_2]\theta[\alpha_3], \quad Q= \theta[\beta_1]\theta[\beta_2]
\theta[\beta_3] 
 \label{pq}
\end{equation}
\end{theorem}
This formula was proven by H.Weber \cite{web879} during his course of deriving special case solutions of the Clebsh problem on the motion of a rigid body in an ideal liquid, and later by O.Bolza in his dissertation 
devoted to the reduction of genus-$2$ holomorphic integrals to elliptic integrals (\cite{bol885}, a shorter version was published in \cite{bol887}).\\  
In more recent times, the problem of a $\theta$-constant representation of $\mathcal{A}$ was discussed within Novikov's program of "effectivization of finite-gap integration formulae", see e.g. Dubrovin \cite{dub981}. E.Belokolos and V.Enolskii \cite{be01} implemented this representation in their approach to the reduction of $\theta$-functional solutions of completely integrable equations to elliptic functions. Nart and Ritzenthaler (\cite{nr17}) used a \emph{Thomae-type} formula for non-hyperelliptic genus-$3$ curves, derived from Weber's formula (\cite{web876} and more recently \cite{fio16}), but did not apply the found $\theta$-constants to the problem of representation of $\mathcal{A}$. An attempt of a generalization of Rosenhain's work can be found in \cite{tak996} and below in Corollaries \ref{cor1} and \ref{cor2}.\\
Some of the build-up of this work can also be found in \cite{er08}, in especially the recovery of Rosenhain's formula by Thomae's second formula. Indeed, it was V. Enolskii, who brought the topic to our attention, and we believe to have generalized their previous contributions.\\[0.25cm]
Below, one of our goals is to elucidate the role that these specific characteristics play in Rosenhain's formula. For that purpose, the next section is dedicated to the first and second Thomae Formulae in higher genera. 
In the 3rd section we go on with the attempt to express the period matrix $\mathcal{A}$ solely by $\theta$-constants, which will then be completed exemplarily in the 4th and 5th section for genus $2$ and $3$ and in doing so we will broader the class of characteristics which fulfill eq. (\ref{rosenhain2}) and its higher-genus analogues.\\
We believe our results will be of general interest as both for theory and for numerical calculations 
of complete hyperelliptic integrals.

\section{Thomae formulae for hyperelliptic curves}

The seminal paper from Thomae [Tho870] is mostly known for the formula relating branch points to even $\theta$-constants of a genus-$g$ hyperelliptic curve $C$. But the paper also contains a formula for non-singular derivated odd $\theta$-constants without a proof. In this section we give an elementary proof..

\subsection{Curve and differentials} Let the curve $C$ be of the form
\begin{align}
y^{2}=\lambda_{2g+2}x^{2g+2}+\lambda_{2g+1}x^{2g+1}+\ldots+\lambda_{0}.
\end{align}
Fix a basis of holomorphic differentials $\mathrm{d}u(P)=(\mathrm{d}u_{1}(P),\ \ldots,\ \mathrm{d}u_{g}(P))^{T},$
\begin{align}
\mathrm{d}u_{i}(P)=\frac{x^{i-1}}{y}\mathrm{d}x,\ i=1,...,g,
\end{align}
and a canonical homology basis $(a,\, b)$. Denote $a$- and $b$-periods
\begin{align}
\mathcal{A}=\left(\oint_{a_k}\mathrm{d}u_{i}\right)_{i,k=1,\ldots,g}\ ,\ B=\left(\oint_{b_{k}}\mathrm{d}u_{i}\right)_{i,k=1,\ldots,g}.
\end{align}
The normalised holomorphic differentials $\mathrm{d}v(P)=(\mathrm{d}v_{1}(P),\ \ldots,\ \mathrm{d}v_{g}(P))$ are defined as
\begin{align}
\mathrm{d}v=\mathcal{A}^{-1}\mathrm{d}u\ \leftrightarrow\oint_{a_{j}}\mathrm{d}v_{i}=\delta_{i,j},\ \oint_{\mathrm{b}_{j}}\mathrm{d}v_{i}=\tau_{i,j},
\end{align}
where the $g\times g$ Riemann matrix $\tau=\mathcal{A}^{-1}\mathcal{B}$ belongs to the Siegel upper half-space $\mathfrak{S}=\{\tau^{T}= \tau,\ {\rm Im}\, \tau>0\}$. Denote $\mathrm{J}\mathrm{a}\mathrm{c}(C)=\mathbb{C}^{g}/\Gamma$ the Jacobi variety of the curve $C$, where $\Gamma=1_{g}\oplus\tau$. Any point $v$ on the Jacobi variety can be represented in the form
\begin{align}
v=\frac12\varepsilon+\frac12 \tau\varepsilon'\ \varepsilon,\ \varepsilon'\in \mathbb{R}^{g}
\label{pointonjacobian}
\end{align}
The vectors $\varepsilon$ and $\varepsilon'$ combine to a $2\times g$ matrix named the characteristic $\left[\varepsilon\right]$ of the point $v$. If $v$ is a half-period then all entries of the characteristic are equal 0 or 1 modulo 2.

\subsection{Theta-functions} Next we introduce in greater detail the Riemann-$\theta$-function $\theta[\varepsilon](z;\tau)$ , $z\in \mathbb{C}^{g}, \tau\in \mathfrak{S}_{g}$
\begin{align}
\begin{split}
\theta[\varepsilon](z;\tau)&=\mathrm{e}^{\frac14\mathrm{i}\pi\varepsilon'^{T}\tau\varepsilon'+\varepsilon'^{T}(z+\frac12\varepsilon)}
\theta\left(z+\frac12\tau\varepsilon'+\frac12\varepsilon\right)\\
&=\sum_{m\in \mathbb{Z}g}\mathrm{exp}\left\{\mathrm{i}\pi\left(m+\frac12\varepsilon'\right)^{T}
\tau\left(m+\frac12\varepsilon'\right)+2\mathrm{i}\pi\left(z+\frac12\varepsilon\right)^{T}
\left(m+\frac12\varepsilon'\right)\right\}.
\end{split}
\end{align}
 with the binary characteristic 
$$\left[\varepsilon\right]=\left[\begin{array}{c}\boldsymbol{\varepsilon}'^T\\ \boldsymbol{\varepsilon}^T\end{array}\right]=\left[\begin{array}{c}\varepsilon_{1}',\ldots,\varepsilon_{g}'\\ \varepsilon_{1},\ldots,\varepsilon_{g}\end{array}\right], \quad \varepsilon_{i}, \varepsilon_{j}'=1\;\text{or}\; 0 $$
It possesses the periodicity property
\begin{align}
\theta[\varepsilon](v+n+\tau n';\tau)=\mathrm{e}^{-2\mathrm{i}\pi n^{\prime T}(v+\frac{1}{2}\tau n')}\mathrm{e}^{\mathrm{i}\pi(n^{T}\varepsilon'-{n'}^T\varepsilon)}\theta[\varepsilon](v;\tau).
\label{periodicity}
\end{align}
The property (\ref{periodicity}) implies
\begin{align}
\theta[\varepsilon](-v;\tau)=\mathrm{e}^{-\pi \mathrm{i}\varepsilon'\varepsilon^{T}}\theta[\varepsilon](v;\tau).
\end{align}
Therefore $\theta[\varepsilon](z;\tau)$ is even if $\varepsilon'\varepsilon^{T}$ is even and odd otherwise. The corresponding characteristic is called even or odd. Among $4^{g}$ characteristics there are $(4^{g}+2^{g})/2$ even and $(4^{g}-2^{g})/2$ odd.\\
Following the notion of Kra\texttt{}zer \cite{kra903}, p. 283, a triplet of characteristics $[\varepsilon_{1}], [\varepsilon_{2}], [\varepsilon_{3}]$ is called {\it azygetic} if
\begin{align}
(-1)^{\varepsilon_{1}'\varepsilon_{1}^{T}+\varepsilon_{2}'\varepsilon_{2}^{T}+\varepsilon_{3}'\varepsilon_{3}^{T}+(\varepsilon_{1}'+\varepsilon_{2}'+\varepsilon_{3}')(\varepsilon_{1}+\varepsilon_{2}+\varepsilon_{3})^{T}}=-1.
\label{azygetic}
\end{align}
and a sequence of $2g+2$ characteristics $[\varepsilon_{1}],\ldots,[\varepsilon_{2g+2}]$ is called a {\it special fundamental system} if the first $g$ characteristics are odd, the remaining are even and any triple of characteristics in it is azygetic.\\
The values $\theta[\varepsilon](0;\tau)=\theta[\varepsilon]$ are called $\theta$-constants. An even characteristic $[\varepsilon]$ is non-singular if $\theta[\varepsilon]\neq 0$, an odd characteristic $[\delta]$ is called non-singular if the derivative $\theta$-constants, $\left.\theta_{k}[\delta]=\partial\theta(z;\tau)/\partial z_{k}\right|_{z=0},$ are not vanishing at least for one index $k.$\\

As it is implied in eq.(\ref{pointonjacobian}) we can identify any branch point $e_{i}$ of the curve $C$ with a half-period,
\begin{align}
\boldsymbol{\mathfrak{A}}_{j}=\int_{P_{0}}^{(e_{j},0)}\mathrm{d}v=\frac12\varepsilon_{j}+\frac12\tau\varepsilon_{j}'
\end{align}
where $P_{0}$ is the base point of the Abel map which is supposed to be a branch point and the integer $2\times 2g$-matrix 
$[\varepsilon]$ is a characteristic of $\mathfrak{A}_{j}$.  We agree to denote with $[\mathfrak{A}_{j}]$ the characteristic of the j-th half-period $\mathfrak{A}_{j}.$

\begin{proposition}{[FK980]} The homology basis $(a,\ b)$ is completely defined by the characteristics $[\mathfrak{A}_{j}]$, $j=1,\ldots,2g+2$. $g$ of these characteristics are odd and the remaining $g+2$ are even. The vector of Riemann constants $K_{P_{0}}$ with a base point $P_{0}$ from the set of branch points is defined in the given homology basis as
\begin{align}
\boldsymbol{K}_{P_0}=\sum_{all\, g\, odd\, [\boldsymbol{\mathfrak{A}}_{j}]}\boldsymbol{\mathfrak{A}}_{j}
\end{align}
\label{chardefhom}
\end{proposition}
\begin{proposition}
Let the $2g+2$ characteristics $[\mathfrak{A}_{j}]$ be ordered into a sequence, for which the first $g$ characteristics are odd and the remaining are even. Then such a system of characteristics is a special fundamental system.
\end{proposition}
\begin{proof}
In the light of Proposition \ref{chardefhom}, it is clear that there are $g$ odd and $g+2$ even characteristics. Now, the first part of the exponent of eq. (\ref{azygetic}) is $0\mod 2$ if none or two of the three characteristics are odd, and it equals $1\mod 2$ if one or three characteristics are odd. The second part of the exponent asks for the parity of a sum of three characteristics. If none or two of them are odd, the sum is odd and hence the said part of the exponent is $1\mod 2$. If one or three characteristics are odd, the sum is even and the second part of the exponent is $0\mod 2$. In total, the exponent is always odd and hence any triple is azygetic.
\end{proof}
Fay (\cite{fay973}, p. 13) describes a one-to-one correspondence between the characteristics $[\varepsilon]$ and the partitions of indices of branch points $\{1,\ldots,2g+2\}$ ,
\begin{align}
\mathcal{I}_{m}\cup \mathcal{J}_{m}=\{i_{1},\ldots,i_{g+1-2m}\}\cup\{j_{1},\ldots,j_{g+1+2m}\},
\end{align}
where $m$ is any integer between $0$ and $[\frac{g+1}{2}]$. Characteristics with given $m$ are defined by the vectors
\begin{align}
\sum_{k\in \mathcal{I}_m}\boldsymbol{\mathfrak{A}}_{i_{k}}-\boldsymbol{K}_{\infty}
=\frac12\boldsymbol{\varepsilon}_{m}+\frac12\tau\boldsymbol{\varepsilon}_{m}'.
\end{align}
Clearly, the following notation for characteristics is useful:
\begin{align}
\left[\varepsilon(\mathcal{I}_m)\right] = \left[\sum_{k\in \mathcal{I}_m}\boldsymbol{\mathfrak{A}}_{i_{k}}-\boldsymbol{K}_{\infty}\right]\equiv \left[\begin{array}{cc}  \varepsilon'_{m,1},\ldots,\varepsilon'_{m,g}\\ \varepsilon_{m,1},\ldots,\varepsilon_{m,g}  \end{array}   \right]\equiv \left\{\mathcal{I}_m\right\}, \quad m=0,1,\ldots\end{align}

$m$ is called the index of speciality of the branch point divisor and we will be interested in the cases $m=0$, that deals with even non-singular $\theta$-constants, and $m=1$, the case of non-singular odd $\theta$-constants. Here, we are considering hyperelliptic curves with a branch point at $\infty$ and we fix in what follows $P_0=\infty$. The defined sets can be written as
$$\mathcal{I}_0=\{ i_1,\ldots, i_g \},\ \mathcal{J}_0=\{ j_1,\ldots,j_{g+1}\}$$ along with the condition: 
$$\mathcal{I}_0\cap \mathcal{J}_0=\emptyset,\ \mathcal{I}_0\cup \mathcal{J}_0=\{1,2,\ldots, 2g+1\}.$$
From the set $\mathcal{I}_0$ $2g$ sets $\mathcal{I}_1$ and $\mathcal{J}_1$ can be defined:
\begin{align} 
\mathcal{I}_{1}^{(n)}=\mathcal{I}_0\backslash \{i_n\} ,\;  \mathcal{J}_{1}^{(n)}=\mathcal{J}_0\cup \{i_n\},\,\ 1\leq n\leq g.
\end{align}
It is convenient to denote the Vandermonde determinants,
\begin{align}
\Delta(\mathcal{I}_{m})=\prod_{i_{k}<i_{l}\in \mathcal{I}_{m}}(e_{i_{l}}-e_{i_{k}}),\ \ \Delta(\mathcal{J}_{m})=\prod_{j_{k}<j_{l}\in J_{m}}(e_{j_{l}}-e_{j_{k}})
\label{Vandermonde}
\end{align}
and we will write as a short form:
\begin{align*}
\nabla(\mathcal{I}_{0})=\Delta(\mathcal{I}_{0})\Delta(\mathcal{J}_{0}),\ \ \nabla(\mathcal{I}_{1})=\Delta(\mathcal{I}_{1})\Delta(\mathcal{J}_{1})
\end{align*}

\subsection{Thomae theorems} We are now in the position of having set up our notation. The odd curve $C$ will be realized as:
\begin{align*}
y^{2}=f(x),\ \ f(x)=(x-e_{1})\cdots(x-e_{2g+1}),\ e_{k}\in \mathbb{C}
\end{align*}
$C$ has a branch point at infinity, $e_{2g+2}=\infty$, and we agree to take away from the products in eq. (\ref{Vandermonde}) all factors containing $e_{2g+2}$.\\
The next theorem is one of the key points of \cite{tho870} and its proof is well-documented in the literature (e.g. \cite{fay973}):
\begin{theorem}[First Thomae theorem] Let $\mathcal{I}_{0}\cup \mathcal{J}_{0}$ be a partition of the set of indices of the finite branch points and $[\varepsilon(\mathcal{I}_{0})]$ the corresponding characteristic. Then
\begin{align}
\theta[\varepsilon(\mathcal{I}_{0})]=\epsilon\left(\frac{\det \mathcal{A}}{2^{g}\pi^{g}}\right)^{1/2}\nabla^{1/4}(\mathcal{I}_{0}),
\label{firstthomae}
\end{align}
where $\epsilon$ is the 8th root of unit, $\epsilon^{8}=1.$
\end{theorem}
To find $\epsilon$, which does not depend on $\tau$ but rather on the ordering of the branch points in $\nabla(\mathcal{I}_0)$, the classical way is to use a diagonal period matrix $\tau$ and use Jacobi's $\theta$-constants relation on the seperated equations. However, we believe the quickest way to determine $\epsilon$ is to compute the $\theta$-constants at a very low precision.\\
There are various corollaries of the Thomae formula (\ref{firstthomae}). The following two are easy to prove. Their formulation is taken from \cite{er08}, but the same result is also topic in \cite{tak996}.
\begin{corollary} \label{cor1} Let $\mathcal S=\{n_1,\ldots,n_{g-1}\}$ and
$\mathcal T=\{m_1,\ldots,m_{g-1}\}$ be two disjoint sets of
non-coinciding integers taken from the set ${\mathcal G}$ of
indices of the finite branch points. Then for any two $k\neq l$
from the set ${\mathcal G} \backslash ({\mathcal S} \cup {\mathcal
T})$ the following formula is valid
\begin{equation}
\frac{e_l-e_m}{e_k-e_m}=\epsilon \frac{\theta^2\{k,{\mathcal
S}\}\theta^2\{k,{\mathcal T}\}} {\theta^2\{l,{\mathcal
S}\}\theta^2\{l,{\mathcal T}\}}, \label{fractions}
\end{equation}
where $m$ is the remaining index when $\mathcal S$, $\mathcal T$,
$k, l$ are taken away from~${\mathcal G}$, and $\epsilon^4 = 1$.
\end{corollary}
\begin{corollary} \label{cor2}
Let ${\mathcal I}_0= \{i_1,\ldots,i_g\}$ and ${{\mathcal J}}_0 =
\{j_1,\ldots j_{g+1}\} $ be a partition of $\mathcal{G}$. Choose $k, n\in
{\mathcal I}_0$ and $i, j \in {\mathcal J}_0$ and define the sets
${\mathcal S}_{k}={\mathcal I}_0 \backslash \{k\}$, ${\mathcal
S}_{k,n}={\mathcal I}_0 \backslash \{k,n\}$, ${\mathcal T}_{i,j}
={\mathcal J}_0\backslash \{i,j\}$. Then
\begin{equation}
\frac{\prod\limits_{j_l\in{\mathcal
J}_0}(e_k-e_{j_l})}{\prod\limits_{i_l\in{\mathcal I}_0, i_l\neq k
}(e_k-e_{i_l})(e_k-e_n)^2} =\pm\frac{\theta^4\{i,{\mathcal
S}_{k}\}\theta^4\{j,{\mathcal S}_{k}\} \theta^4\{n,{\mathcal
T}_{i,j}\}}{  \theta^4\{i,j,{\mathcal S}_{k,n} \}
\theta^4\{i,{\mathcal T}_{i,j}\} \theta^4\{j,{\mathcal T}_{i,j}
\}}. \label{fractions1}
\end{equation}
\end{corollary}
One can assure oneself of the correctness of these corollaries by a straightforward calculation and use of Thomaes first theorem.\\[0.5cm]
Thomae's paper contains also an important theorem describing non-singular derivated odd $\theta$-constants. It is this, which is most significant for the course of the paper at hand:
\begin{theorem}[Second Thomae theorem] Let $\mathcal{I}_{1}^{\left(n\right)}\cup \mathcal{J}_{1}^{\left(n\right)}$ be a partition of the set of indices of the finite branch points. Then
\begin{align}
\theta_j[\varepsilon(\mathcal{I}_{1}^{\left(n\right)})]=\epsilon\left(\frac{\det \mathcal{A}}{2^{g+2}\pi^{g}}\right)^{1/2}\nabla(\mathcal{I}_{1}^{\left(n\right)})^{1/4}\sum_{i=1}^{g}\mathcal{A}_{j,i}s_{g-i}(\mathcal{I}_{1}^{\left(n\right)}) ,\ j=1\ldots g, 
\label{secondthomae}
\end{align}
where $\epsilon^{8}=1$ and $s_{j}(\mathcal{I}_{1}^{\left(n\right)})$ is the elementary symmetric function of degree $j$ built in the branch points $e_{i}, i\in \mathcal{I}_{1}^{\left(n\right)},$ and alternated in the sign.
\label{secondthomaetheorem}
\end{theorem}
We give here an elementary proof of this theorem. For this we first examine a helpful lemma.
\begin{lemma} Let $\mathfrak{A}_{k}$ be the Abelian image of the branch point $e_{k}$ and $[\mathfrak{A}_{k}]$ its characteristic. Let $v=\sum_{j=1}^{g}\int_{\infty}^{P_{j}}\mathrm{d}v=\mathcal{A}^{-1}u$ with $P_{j}=(x_{j}, y_{j})$. Then at $k=1,\ldots,2g+1$
\begin{align}
\left(\frac{\theta[\mathfrak{A}_{k}](v-K_{\infty})}{\theta(v-K_{\infty})}\right)^{2}=\frac{\epsilon_{4}}{\sqrt{f'(e_{k})}}\prod_{j=1}^{g}(e_{k}-x_{j}),
\label{thetaquotlemma}
\end{align}
with $\epsilon_{4}^{4}=1.$
\end{lemma}

\begin{proof} Consider the expression
\begin{align}
F(P_{1},\ldots,P_{g})=\frac{\theta^{2}\left(\int_{\infty}^{(e_{k},0)}\mathrm{d}v+v-K_{\infty}\right)}{\theta^{2}(v-K_{\infty})}
\end{align}
As the function of $\overline{P}_{1}$ (here $\overline{P}=(x,-y)$ whilst $P=(x,y)$) it has, according to the Riemann vanishing theorem, zeros of second order in the points $(e_{k},0)$, $P_{2},\ldots,P_{g}$ and poles of second order in $\infty, P_{2},\ldots,P_{g}$. Thus $F(P_{1},\ldots,P_{g})\sim c_{1}(x_{1}-e_{k})$. Considering in the same way other the variables $P_{2},\ldots,P_{g}$ we conclude
\begin{align}
\left(\frac{\theta[\mathfrak{A}_{k}](v-K_{\infty})}{\theta(v-K_{\infty})}\right)^{2}=c(e_{k}-x_{1})\cdots(e_{k}-x_{g})
\label{lemmawitha}
\end{align}
To find the constant $c$ we use (\ref{firstthomae}). Therefore we fix $v$ at some branch points $P_{i_1},\ldots,P_{i_g}$, $P_{i_j}=(e_{i_j},y(e_{i_j}))$, and rewrite eq. (\ref{lemmawitha}) using the $\varepsilon$-notation for the characteristics:
$$
\left(\frac{\theta[\mathfrak{A}_{k}](v-K_{\infty})}{\theta(v-K_{\infty})}\right)^{2}=\frac{\theta^2[\varepsilon_{ki_1\ldots i_g}]}{\theta^2[\varepsilon_{i_1\ldots i_g}]}=\epsilon_4\sqrt{\frac{\nabla(\{k,i_1,\ldots,i_g\})}{\nabla(\{i_1,\ldots,i_g\})}}=\frac{\epsilon_4}{\sqrt{f'(e_k)}}\cdot(e_{k}-e_{i_1})\cdots(e_{k}-e_{i_g}).
$$ 
\end{proof}
\begin{proof} Coming back to the proof of Theorem \ref{secondthomaetheorem} we introduce the functions
\begin{align}
\begin{split}
&F(x)=\prod_{k=1}^{g}(x-x_{k}),\\
&F_{i}(x)=F(x)/(x-x_{i})=x^{g-1}+s_{1}^{(i)}x^{g-2}+\ldots+s_{g-1}^{(i)},\ i=1,\ldots,g,
\end{split}
\end{align}
where $s_{j}^{(i)}$ are the elementary symmetric functions of order $j$ built in the elements $\{x_{1},\ldots,x_{g}\}/\{x_{i}\}$ and alternated in sign, namely,
\begin{align*}
\begin{split}
&s_{0}^{(i)}=1\\
&s_{1}^{(i)}=-x_{p}-\ldots,\ \ p\in\{1,\ldots,g\}/\{i\}\\
&s_{2}^{(i)}=x_{p}x_{q}+\ldots,\ \ p,q\in\{1,\ldots,g\}/\{i\}\\
&\vdots\\
&s_{g-1}^{(i)}=\pm\prod x_{r},\ \ r\in\{1,\ldots,g\}/\{i\}
\end{split}
\end{align*}
First, we use these functions to compute the Jacobian $\frac{\partial x}{\partial v}$. Differentiating the Abel map,
\begin{align*}
\begin{split}
\int_{\infty}^{x_1(u_1,\ldots,u_g)}\frac{\d x}{y}+&\ldots+\int_{\infty}^{x_g(u_1,\ldots,u_g)}\frac{\d x}{y}=u_1\\
&\vdots\\
\int_{\infty}^{x_1(u_1,\ldots,u_g)}\frac{x^{g-1}\d x}{y}+&\ldots+\int_{\infty}^{x_g(u_1,\ldots,u_g)}\frac{x^{g-1}\d x}{y}=u_g,
\end{split}
\end{align*}
with respect to $u_1$ we get
\begin{align}
\begin{split}
\frac1{y_1}\frac{\partial x_1}{\partial u_1}+&\ldots+\frac1{y_g}\frac{\partial x_g}{\partial u_1}=1\\
&\vdots\\
\frac{x_1^{g-1}}{y_1}\frac{\partial x_1}{\partial u_1}+&\ldots+\frac{x_g^{g-1}}{y_g}\frac{\partial x_g}{\partial u_1}=0
\end{split}
\end{align}
and similar for the other variables. Solving these equations with respect to $\frac{\partial x_i}{\partial u_j}$, we arrive at
\begin{align}
\frac{\partial x}{\partial v}=\frac{\partial\left(x_1,\ldots,x_g\right)}{\partial\left(v_1,\ldots,v_g\right)}=\mathcal{A}^T\frac{\partial\left(x_1,\ldots,x_g\right)}{\partial\left(u_1,\ldots,u_g\right)}=\mathcal{A}^T\left(\frac{y_is_{g-j}^{(i)}}{F'(x_i)}\right)_{i,j=1,\ldots,g}
\end{align}
Aside from that, we compute the derivative of eq. (\ref{thetaquotlemma}):
\begin{align}
\frac{\partial}{\partial v_i}\frac{\theta[\mathfrak{A}_{k}](v-K_{\infty})}{\theta(v-K_{\infty})}=\frac{\epsilon}{f'(e_{k})^{1/4}}\frac{1}{2\sqrt{\prod_{j=1}^{g}(e_{k}-x_{j})}}\frac{\partial}{\partial v_i}\prod_{j=1}^g\left(e_k-x_j\right),\quad i=1,\ldots,g
\end{align}
which can be processed for our purposes to:
\begin{align}
\frac{\partial}{\partial v_i}\prod_{j=1}^g\left(e_k-x_j\right)=-\sum_{j=1}^g\frac{\partial x_j}{\partial v_i}F_j(e_k)=-\left(\frac{\partial x_1}{\partial v_i},\ldots,\frac{\partial x_g}{\partial v_i}\right)\cdot\left(\begin{array}{c}F_1(e_k)\\ \vdots\\ F_g(e_k) \end{array}\right)
\label{Jacobgradient}
\end{align}
To write this relation for $\theta$-constants, we proceed like in the previous proof and fix $v$ at certain branch points: $x_j=e_l$, $j=1,\ldots,g$, $l=1,\ldots,2g+1$. Again we can adopt the $\varepsilon$-notation and write:
\begin{align}
\left(\begin{array}{c}
\frac{\partial}{\partial v_{1}}\\
\vdots\\
\frac{\partial}{\partial v_{g}}
\end{array}\right)\frac{\theta[\varepsilon_{k;l_1\ldots l_g}]}{\theta[\varepsilon_{l_1\ldots l_g}]}=\frac{\epsilon}{f'(e_{k})^{1/4}}\frac{1}{2\sqrt{\prod_{j=1}^{g}(e_{k}-e_{l_j})}}\mathcal{A}^T\left(\frac{y_is_{g-j}^{(i)}}{F'(e_{l_i})}\right)^T\cdot \left(\begin{array}{c}F_1(e_k)\\ \vdots\\ F_g(e_k) \end{array}\right).
\label{interimstep}
\end{align}
where the minus sign of eq. (\ref{Jacobgradient}) was absorbed in $\epsilon$.
Of course, the different $y_i=\sqrt{\prod_{j=1}^{2g+1}(x_i-e_j)}$ will become zero if $x_i=e_{l_i}$ and hence the whole expression cancels \emph{unless} $e_k$ coincides with with that specific $e_{l_i}$ so that these factors in the numerator and denominator can cancel. Without loss of generality we choose $k=l_g$ and hence we have $[\varepsilon_{k;i_1\ldots i_g}]=[\varepsilon_{i_1\ldots i_{g-1}}]$. These $g-1$ elements shall now constitute the set $\mathcal{I}_1$ for they form all the non-singular and odd characteristics. The characteristics from the left hand side's denominator, $[\varepsilon_{i_1\ldots i_{g}}]$, we merge into the set $\mathcal{I}_0$. 
$\theta_k[\varepsilon_{\mathcal{I}_1}]$ does not vanish, but $\theta_k[\varepsilon_{\mathcal{I}_0}]$ does. The derivative hence becomes:
\begin{align*}
\left(\begin{array}{c}
\frac{\partial}{\partial v_{1}}\\
\vdots\\
\frac{\partial}{\partial v_{g}}
\end{array}\right)\frac{\theta[\varepsilon_{l_1\ldots l_{g-1}}]}{\theta[\varepsilon_{l_1\ldots l_g}]}=\frac1{\theta[\varepsilon_{\mathcal{I}_0}]}\left(\begin{array}{c}\theta_1[\varepsilon_{\mathcal{I}_1}]\\ \vdots \\ \theta_g[\varepsilon_{\mathcal{I}_1}] \end{array}\right)
\end{align*}
However, on the right hand side of eq. (\ref{interimstep}) all remaining $y_i$ cancel and the residual zeros of $\sqrt{\prod_{j=1}^{g}(e_{k}-e_{l_j})}$ will be canceled by the factors of $F_i(e_k)$.\\
Plugging all this together, we get:
\begin{align}
\left(\begin{array}{c}
\theta_{1}[\varepsilon_{\mathcal{I}_1}]\\
\vdots\\
\theta_{g}[\varepsilon_{\mathcal{I}_1}]
\end{array}\right)=\epsilon\,\theta[\varepsilon_{\mathcal{I}_0}]\,\sqrt[4]{\chi_{k}}\,\mathcal{A}^{T}\left(\begin{array}{c}
s_{g-1}(\mathcal{I}_{1})\\
\vdots\\
 s_{1}(\mathcal{I}_{1})\\
1
\end{array}\right),
\end{align}
where it is denoted
\begin{align}
\chi_{k}=\frac{\prod_{j\in \mathcal{J}_{1},j\neq k}(e_{k}-e_{j})}{\prod_{i\in \mathcal{I}_{1}}(e_{k}-e_{i})},\ k=1,\ldots,2g+1,
\label{chidef}
\end{align}
with $\mathcal{J}_1$ the opposite partition of $\mathcal{I}_1$ as usual.\\
On $\theta[\varepsilon_{\mathcal{I}_0}]$ we can use Thomae's first theorem (\ref{firstthomae}). Recognizing, that $\nabla(\mathcal{I}_0)\cdot \chi_k=\nabla(\mathcal{I}_1)$, we arrive at the statement of the theorem.
\end{proof}
\begin{bf} Example: The genus-$1$ case\end{bf}\\
Let $C$ be the Weierstrass cubic,
\begin{align*}
y^{2}=4(x-e_{1})(x-e_{2})(x-e_{3}),\ \ e_{1}+e_{2}+e_{3}=0
\end{align*}
In this case we have: $\mathcal{I}_{1}^{\left(1\right)}=\emptyset$, $\mathcal{J}_{1}^{\left(1\right)}=\{1,2,3\}$,\\ $\Delta(\mathcal{I}_{1}^{\left(1\right)})=1$, $\Delta(\mathcal{J}_{1}^{\left(1\right)})=(e_{1}-e_{2})(e_{1}-e_{3})(e_{2}-e_{3})$,\\ $s_{0}(\mathcal{I}_{1}^{\left(1\right)})=1$ and $\varepsilon(\mathcal{I}_{1}^{\left(1\right)})=-K_{\infty}=[_{1}^{1}]$\\[0.5cm]
Using further \cite{be955}, vol 3, Sect 13.20 
\begin{align*}
(e_{1}-e_{2})^{1/2}=\frac{\pi}{2\omega}\vartheta_{4}^{2}(0),\ \ (e_{1}-e_{3})^{1/2}=\frac{\pi}{2\omega}\vartheta_{3}^{2}(0),\ \ (e_{2}-e_{3})^{1/2}=\frac{\pi}{2\omega}\vartheta_{2}^{2}(0),
\end{align*}
then (\ref{secondthomae}) takes the form of the Jacobi derivative relation
\begin{align}
\theta_{1}'(0)=\pi\vartheta_{2}(0)\vartheta_{3}(0)\vartheta_{4}(0).
\label{Jacobideriv}
\end{align}
\\[0.5cm]

\begin{bf}Example: The genus-$g$ case \end{bf}\\ 
With much the same method, one can obtain a generalization of eq. (\ref{Jacobideriv}) to arbitrary genus, which is known as the Riemann-Jacobi-formula (\cite{fay979}). For that to formulate we introduce the $g+1$ sets 
\begin{align}
\mathcal{T}_n=\mathcal{J}_0\backslash\{j_n\},\,\ 1\leq n\leq g+1,
\label{T-sets}
\end{align}
and also write $\mathcal{T}_0=\mathcal{J}_0$. The characteristics $[\varepsilon(\mathcal{I}_{1}^{(n)})]$ are non-singular and odd, the characteristics $[\varepsilon(\mathcal{T}_n)]$ non-singular and even.\\
Further, we need the Jacobi matrix $J$:
\begin{equation}
\label{jacobimatrix} 
J=\left.\frac{\partial(\theta[\varepsilon(\mathcal{I}_{1}^{(1)})](v),\ldots,\theta[\varepsilon(\mathcal{I}_{1}^{(n)})](v))}{\partial(v_{1},\ldots,v_{g})}\right|_{v=0}.
\end{equation}
Then, the following relation is valid:
\begin{equation}
\det J=\pm\pi^g\prod_{n=0}^{g+1}\theta[\varepsilon(\mathcal{T}_n)]
\end{equation}
This is a long-known result. See e.g. \cite{er08} for a proof within the methods described here. There, we find also a useful matrix notation for the second Thomae formula, which we will adopt in the next section:

\subsection{Matrix form of the Second Thomae formula}
The formula (\ref{secondthomae}) can be written in matrix form. With all the definitions above, one immediately recognises that this comes to:
\begin{align}
J^T=\epsilon\left(\frac{\mathrm{Det}\mathcal{A}}{2^{g+2}\pi^{g}}\right)^{1/2}
\mathcal{A}^{T}. \mathcal{S} .\mathcal{D},
\label{secondthomaematrix}
\end{align}
where  $\mathcal{S}$ is the invertable matrix
\begin{align}
\mathcal{S}=(s_{i}^{(j)})_{i,j=1,\ldots,g}, \label{smatrix}
\end{align}
and $\mathcal{D}$ is a diagonal matrix:
\begin{align*}
\mathcal{D}=\text{Diag} \left[\nabla(\mathcal{I}_{1}^{(1)})^{1/4},\ldots,\nabla(\mathcal{I}_{1}^{(g)})^{1/4}\right].
\end{align*}
This formula can also be rewritten as
\begin{align}
\mathcal{A}=\epsilon\sqrt{\frac{2^{g+2}\pi^{g}}{\det\mathcal{A}}}{\mathcal{S}^{-1}}^{T}.
\mathcal{D}^{-1}.J
\label{Thomaematrixstep1}
\end{align}
This can be treated further. To do that we write our given curve of genus $g$ (with a branch point at infinity) in the form:
\begin{equation} y^2= \phi(x) \psi(x)  \end{equation}
with 
\begin{equation}
\phi(x)=\prod_{i_k \in \mathcal{I}_0} (x-e_{i_k})  ,\quad  \psi(x)=\prod_{j_l \in \mathcal{J}_0} (x-e_{j_l}),
\end{equation}
After inverting eq. (\ref{Thomaematrixstep1}), we get:
\begin{align}
\mathcal{A}^{-1}=\epsilon\sqrt{\frac{\det\mathcal{A}}{2^{g+2}\pi^{g}}}\frac{\mathrm{Adj} J.\mathcal{D}.\mathcal{S}^T}{\det{J}}.
\end{align}
Next, we want to use the Riemann-Jacobi formula on $\det J$. For that we need the $g+1$ sets $\mathcal{T}_k=\mathcal{J}_0\backslash\{j_k\}$, so that we can write:
\begin{align}
\mathcal{A}^{-1}=\epsilon\sqrt{\frac{\det\mathcal{A}}{2^{g+2}\pi^{g}}}\frac{1}{\pi^g\theta[\varepsilon(\mathcal{T}_0)]\Theta_{\mathcal{I}_0}}\mathrm{Adj} J.\mathcal{D}.\mathcal{S}^T.
\end{align}
with
\begin{align}
\Theta_{\mathcal{I}_0}=\prod_{l=1}^{g+1}\theta[\varepsilon(\mathcal{T}_l)].
\label{Bigtheta0}
\end{align}
Note, that $\mathcal{T}_0$ is excluded from the product, which will be useful later, as well as the label of $\Theta$. We can process $\theta[\varepsilon(\mathcal{T}_0)]$ further with the help of the first Thomae theorem, so that most of the prefactors cancel and only a $\nabla^{1/4}(\mathcal{I}_0)$ remains in the denominator. As $\sqrt{\det \mathcal{A}}$ cancels, we get also rid of one possible source of a prefactor. Now using the previous found relation (\ref{chidef})
\begin{align}
 \frac{\Delta(\mathcal{I}_1^{(n)})}{\Delta(\mathcal{I}_0)} \; \frac{\Delta(\mathcal{J}_1^{(n)})}{\Delta(\mathcal{J}_0)} 
= \frac{\psi(e_{i_n})}{\phi'(e_{i_n})}=\chi_{i_n}, \quad n=1,\ldots,g  
\end{align}
and defining
\begin{align}
\mathcal{D}_1&=\mathrm{Diag}\left[\sqrt[4]{  \chi_{i_1} },  \ldots,   
\sqrt[4]{\chi_{i_g}}\right] 
\end{align}
we get the final form:
\begin{align}
\mathcal{A}^{-1}=\frac{\epsilon}{2\pi^g \Theta_{\mathcal{I}_0}}\mathrm{Adj} J.\mathcal{D}_1.\mathcal{S}^{T}.
\label{generalrosenhain}
\end{align}
Of course, one can consider also the not-inversed, original period matrix $\mathcal{A}$, and using the same steps on $\mathcal{D}$ as before, we can rewrite eq. (\ref{Thomaematrixstep1}) as:
\begin{align}
\mathcal{A}=\epsilon\frac{2}{\theta[\varepsilon(\mathcal{J}_0)]}{\mathcal{S}^{-1}}^{T}.\mathcal{D}_1^{-1}.J.
\label{notinverseA}
\end{align}
But we decided to work primarily on the inversed matrix, because this is, what Rosenhain's formula gives us. Another advantage is that we can quickly recover and generalize Bolza's formula. For that purpose we write our result in the following way:
\begin{proposition}
Let $\mathcal{C}$ be a hyperelliptic curve of genus $g$ with one branch point at infinity. 
Let $\mathcal{I}_0\cup \mathcal{J}_0 = \{ i_1,\ldots, i_g \} \cup \{ j_1,\ldots, j_{g+1}\}$ be a
partition of $2g+1$ indices of branch points. Then the columns $\boldsymbol{U}_m$ of the matrix $\mathcal{A}^{-1}$  are of the form
\begin{align}
\boldsymbol{U}_m=\frac{ 
\epsilon}{2\pi^g \Theta_{\mathcal{I}_0}} 
{\mathrm{Adj}(J)}
\left( \begin{array}{c} s^{i_1}_{g-m} \sqrt[4]{\chi_{i_1}} \\ \vdots \\ 
                                           s^{i_g}_{g-m}\sqrt[4]{\chi_{i_g}}    \end{array}   \right), 
\quad m=1, \ldots,g\label{finalanswer1}
\end{align}
\end{proposition}

\subsection{Bolza formulae}
Let $\partial_{\boldsymbol{U}_k}$ be the directional derivative along the vector $\boldsymbol{U}_k$ at zero argument:
\begin{align*}
\left.\partial_{\boldsymbol{U}_k} f(\boldsymbol{v})=\sum_{j=1}^gU_{k,j}\frac{\partial}{\partial v_j}f(\boldsymbol{v})\right\vert_{\boldsymbol{v}=0}\ \quad ,k=1,\ldots,g.
\end{align*}
For a genus-$2$ curve with branch points $e_1,\ldots,e_{2g+1}$, Bolza (\cite{bol887}) found (without proof!) that 
\begin{align}
e_i=-\frac{\partial_{\boldsymbol{U}_1}\theta[\boldsymbol{\mathfrak{A}}_i+\boldsymbol{K}_{P_0}]}{\partial_{\boldsymbol{U}_2}\theta[\boldsymbol{\mathfrak{A}}_i+\boldsymbol{K}_{P_0}]}.
\end{align}
With the help of eq. (\ref{finalanswer1}) we are in the position to highly generalize this result. As we have done before, we choose $P_0=\infty$ and keep the notation of $\varepsilon$ instead of $\boldsymbol{\mathfrak{A}}$. For a general hyperelliptic curve of genus $g$ we consider the expression
\begin{align*}
\frac{\partial_{\boldsymbol{U}_m}\theta[\varepsilon_{\mathcal{I}_1^{(j)}}]}{\partial_{\boldsymbol{U}_n}\theta[\varepsilon_{\mathcal{I}_1^{(j)}}]}, \quad m,n=1,\ldots,g
\end{align*}
There are $g$ different sets $\mathcal{I}_{1}^{(j)}=\mathcal{I}_0\backslash \{i_j\}$, which also constitute the matrices $J$ and $\mathcal{S}$. Inserting eq. (\ref{finalanswer1}) into this expression, we find that all $\theta$-constants cancel out, as well as the prefactors of $\boldsymbol{U}$ and the 4th. roots. We arrive at the 
\begin{corollary}
Let $\partial_{\boldsymbol{U}_k}$ be directional derivatives, $\mathcal{I}_{1}^{(j)}$ $g$ sets of $g-1$ branch point indices and $s_k(\mathcal{I}_{1}^{(j)})$ alternating elementary symmetric functions of order $k$ over the elements $e_i$, $i\in\mathcal{I}_{1}^{(j)}$. Then the following generalized Bolza formulae are valid:
\begin{align}
\frac{s_{g-m}(\mathcal{I}_{1}^{(j)})}{s_{g-n}(\mathcal{I}_{1}^{(j)})}=\frac{\partial_{\boldsymbol{U}_m}\theta[\varepsilon_{\mathcal{I}_1^{(j)}}]}{\partial_{\boldsymbol{U}_n}\theta[\varepsilon_{\mathcal{I}_1^{(j)}}]}, \quad m,n=1,\ldots,g
\end{align}
\end{corollary} 
\begin{bf} Example: genus $3$\end{bf}\\
Take $\mathcal{I}_0=\{1,2,3\}$ and hence $\mathcal{I}_1^{(1)}=\{2,3\}$, $\mathcal{I}_1^{(2)}=\{1,3\}$ and $\mathcal{I}_1^{(3)}=\{1,2\}$. We find :
\begin{align}
\begin{split}
&\frac{\partial_{\boldsymbol{U}_1}\theta[\varepsilon_{12}]}{\partial_{\boldsymbol{U}_3}\theta[\varepsilon_{12}]}=\frac{s_2(e_1,e_2)}{s_0(e_1,e_2)}=e_1\cdot e_2\\
&\frac{\partial_{\boldsymbol{U}_2}\theta[\varepsilon_{12}]}{\partial_{\boldsymbol{U}_3}\theta[\varepsilon_{12}]}=\frac{s_1(e_1,e_2)}{s_0(e_1,e_2)}=-e_1-e_2,
\end{split}
\end{align}
and all other combinations of $m$ and $n$ can be derived from these both.

\section{A general $\theta$-constant form of $\mathcal{A}^{-1}$}
Though eq. (\ref{finalanswer1}) gives us a good tool, our final goal is to completely express $\mathcal{A}^{-1}$ with $\theta$-constants for those cases where only $\tau$ is known. Thus, we want to work more on $\chi_k$. We can achieve that by the use of eq. (\ref{thetaquotlemma}).
\begin{theorem}
Let $\mathcal{I}_0=\{n,i_1,\ldots,i_{g-1}\}$ and $\mathcal{J}_0=\{j_1,\ldots,j_{g+1}\}$ be a partition of branch points, such that $y^2=\phi(x)\psi(x)$ with 
\begin{align*}
\phi(x)=\prod_{i \in \mathcal{I}_0} (x-e_{i})  ,\quad  \psi(x)=\prod_{j \in \mathcal{J}_0} (x-e_{j}).
\end{align*}
Let further $\mathcal{I}_1=\mathcal{I}_0\setminus \{n\}$. 
For $\chi_n=\frac{\psi(e_n)}{\phi'(e_n)}$, $n\in\mathcal{I}_0$, we find:
\begin{align}
\sqrt[4]{\chi_n}=\epsilon\sqrt[g-1]{\frac{\Theta_{\mathcal{I}_1}}{\Theta_{\mathcal{I}_0}}}\cdot\sqrt[2g-2]{\prod_{i\in\mathcal{I}_0,i\neq n}\left(e_n-e_i\right)},
\end{align}
where 
\begin{align*}
\Theta_{\mathcal{I}_0}=\prod_{j\in\mathcal{J}_0}\theta[\varepsilon(\mathcal{J}_0\setminus\{j\})],\quad
\Theta_{\mathcal{I}_1}=\prod_{j\in\mathcal{J}_0}\theta[\varepsilon(\mathcal{I}_1\cup\{j\})]
\end{align*}
\label{generalchi}
\end{theorem}
\begin{proof}
Take eq. (\ref{thetaquotlemma}) and evaluate $v$ at the branch points $e_{j_1},\ldots,e_{j_g}$:
\begin{align*}
\begin{split}
\frac{\theta^2[\varepsilon_{nj_1\ldots j_g}]}{\theta^2[\varepsilon_{j_1\ldots j_g}]}&=\frac{\epsilon_4}{f'(e_n)}\prod_{l=1}^g\left(e_n-e_{j_l}\right)\\
&=\epsilon_4\sqrt{\frac{(e_n-e_{j_1})\cdots(e_n-e_{j_g})}{(e_n-e_{i_1})\cdots(e_n-e_{i_{g-1}})\cdot(e_n-e_{j_{g+1}})}}.
\end{split}
\end{align*}
Squaring this and iterating the procedure for every left-over $j_{g+1}$ we get:
\begin{align}
\begin{split}
\prod_{j\in\mathcal{J}_0}\frac{\theta^4[\varepsilon_{\{n\}\cup\mathcal{J}_0\setminus\{j\}}]}{\theta^4[\varepsilon_{\mathcal{J}_0\setminus\{j\}}]}\equiv\frac{\Theta^4_{\mathcal{I}_1}}{\Theta^4_{\mathcal{J}_0}}&=\pm\frac{\prod_{j\in\mathcal{J}_0}\left(e_n-e_j\right)^{g-1}}{\prod_{i\in\mathcal{I}_0,i\neq n}\left(e_n-e_i\right)^{g+1}}\\&=\pm\frac{\chi_n^{g-1}}{\prod_{i\in\mathcal{I}_0,i\neq n}\left(e_n-e_i\right)^{2}}.
\end{split}
\label{chiproof}
\end{align}
This equality comes from the fact that there are $g$ times $g+1$ terms in the numerator and every linear factor occurs $g$ times, but is canceled once by the denominator. The residual parts fit the definition of $\chi_n$.\\ 
Finally, we recognize $\varepsilon(\{n\}\cup\mathcal{J}_0\setminus\{j\})=\varepsilon(\{n\}\cup\mathcal{I}_0\cup\{j\})=\varepsilon(\mathcal{I}_1\cup\{j\})$ to arrive at the definition of $\Theta_{\mathcal{I}_1}$. 
\end{proof}
\begin{bf} Note: \end{bf} In eq. (\ref{chiproof}) are as much factors in the numerator as in the denominator. Therefore, we can interchange the ordering of the $e_n$ and $e_j$ without changing the global prefactor $\epsilon$, if one simultaneously changes the ordering of the $e_n$ and $e_i$ in the denominator. 

\section{Genus 2: Recovery of Rosenhain's formula}
\subsection{The Rosenhain derivatives}
Consider the case $g = 2$ and the curve given as
\begin{equation}
y^2=(x-e_1)(x-e_2)(x-e_3)(x-e_4)(x-e_5)\equiv f(x) \label{curve2}
\end{equation}

In the homology basis drawn on Fig. \ref{figure-1} we have
\begin{align}
\begin{split}
&[\boldsymbol{\mathfrak{A}}_1]=\left[ \begin{array}{cc}1&0\\0&0\end{array} \right],
\quad[\boldsymbol{\mathfrak{A}}_2]=\left[ \begin{array}{cc}1&0\\1&0\end{array} \right],
\quad[\boldsymbol{\mathfrak{A}}_3]=\left[ \begin{array}{cc}0&1\\1&0\end{array} \right],\\
&[\boldsymbol{\mathfrak{A}}_4]=\left[ \begin{array}{cc}0&1\\1&1\end{array} \right],
\quad[\boldsymbol{\mathfrak{A}}_5]=\left[ \begin{array}{cc}0&0\\1&1\end{array} \right],
\quad[\boldsymbol{\mathfrak{A}}_6]=\left[ \begin{array}{cc}0&0\\0&0\end{array} \right]
\end{split}\label{sfs2}
\end{align}
The characteristic of the vector of Riemann constants reads
\begin{equation}
[\boldsymbol{K}_{\infty}] =[\boldsymbol{\mathfrak{A}}_2]+ [\boldsymbol{\mathfrak{A}}_4] \equiv [\boldsymbol{\mathfrak{A}}_1]+ [\boldsymbol{\mathfrak{A}}_3] +[\boldsymbol{\mathfrak{A}}_5] = \left[\begin{array}{cc}
1 & 1\\
0 & 1
\end{array}\right]
\label{Kinf}
\end{equation}

The characteristics in question here are: 
 \begin{align}\begin{split}{}
[ \varepsilon_i] = \left[\boldsymbol{\mathfrak{A}}_i\right] - \left[\boldsymbol{K}_{\infty}  \right], \quad &i=1,\ldots, 6\\
[\varepsilon_{ij}]= \left[\boldsymbol{\mathfrak{A}}_i\right] +\left[\boldsymbol{\mathfrak{A}}_j\right] - \left[\boldsymbol{K}_{\infty}  \right], \quad &i,j=1,\ldots,6, \, i\neq j,
\end{split}
\end{align}
and analogously for three indices, if necessary. The first line represents the $6$ odd characteristics, the second line the $10$ even characteristics. Due to the addition$\mod 2$ one easily sees that $[\varepsilon_2]=[\boldsymbol{\mathfrak{A}}_4]$, $[\varepsilon_4]=[\boldsymbol{\mathfrak{A}}_2]$ and $[\varepsilon_{24}]=2\cdot[\boldsymbol{\mathfrak{A}}_2]+2\cdot[\boldsymbol{\mathfrak{A}}_4]=[_{00}^{00}]=[\boldsymbol{\mathfrak{A}}_6]$\\
One also has to hold in mind, that the sum of all characteristics $\mathfrak{A}_i$ is zero, so that $2$-indexed $\varepsilon$ and $3$-indexed $\varepsilon$ can be interchanged (as shown for instance in eq. (\ref{Kinf})).\\

We are now in the position to exemplary investigate the sets $\mathcal{T}_l$ of eq. (\ref{T-sets}) and henceforward $\Theta_{\mathcal{I}_0}$ of eq. (\ref{Bigtheta0}). We therefore split $f(x)=\phi(x)\psi(x)$ like before and specify $\phi$ and $\psi$ by fixing $\mathcal{I}_0=\{1,2\}$ and $\mathcal{J}_0=\{3,4,5\}$, so that:
$$\mathcal{T}_1=\{3,4\},\ \ \mathcal{T}_2=\{3,5\},\ \ \mathcal{T}_3=\{4,5\}.$$
The already defined quantity $\Theta_{\mathcal{I}_0}$ becomes:
\begin{align}
\Theta_{\{1,2\}}=\theta[\varepsilon_{34}]\theta[\varepsilon_{35}]\theta[\varepsilon_{45}].
\end{align}
This choice of the sets leads us directly to the following Rosenhain derivative formula as a consequence of the Riemann-Jacobi-formula:
\begin{equation}
\theta_1 [\varepsilon_2 ]\theta_2[\varepsilon_1]-
\theta_1 [\varepsilon_1 ]\theta_2[\varepsilon_2]
=\pi^2\theta [\varepsilon_{34}]\theta [\varepsilon_{35}]\theta [\varepsilon_{45}] \theta [\varepsilon_{345}]\equiv \pi^2\Theta_{\{1,2\}} \theta [\varepsilon_{345}]
\label{Rosder}
\end{equation}

In general, for the different choices of $\varepsilon_i$, $\varepsilon_j$ as odd characteristics Riemann-Jacobi gives us $\binom{2g+1}{g}=\binom{5}{2}=10$ different Rosenhain derivative formulae (up to a minus sign due to the antisymmetry of the determinant), and $5$ more, if one includes $\varepsilon_6\equiv K_{\infty}$. These last $5$ equations belong to the $5$ possible sets $\mathcal{I}_0=\{i,6\}$, which are not covered by our notation, though they are valid anyway. All these $15$ relations are shown in the Appendix $A$ with their correct ordering to fix the sign.\\
For any triple $\{i,j,k\}\subset\{1,\ldots,6\}$ we can regard the three Rosenhain derivative formulae belonging to the sets $\mathcal{I}_0^1=\{i,j\}$, $\mathcal{I}_0^2=\{i,k\}$ and $\mathcal{I}_0^3=\{j,k\}$. Among the even characteristics on the right-hand-side of them there will be precisely one characteristic $\varepsilon_{lmp}$, $\{l,m,p\}=\{1,\ldots,6\}\setminus\{i,j,k\}$, which appears in all three formulae. We therefore write:
\begin{align}
\begin{split}
&\theta_1[\varepsilon_i]\theta_2[\varepsilon_j]-\theta_1[\varepsilon_j]\theta_2[\varepsilon_i]=\pi^2\theta[\varepsilon_{lmp}]\Theta_{\{i,j\}}\\
&\theta_1[\varepsilon_j]\theta_2[\varepsilon_k]-\theta_1[\varepsilon_k]\theta_2[\varepsilon_j]=\pi^2\theta[\varepsilon_{lmp}]\Theta_{\{j,k\}}\\
&\theta_1[\varepsilon_i]\theta_2[\varepsilon_k]-\theta_1[\varepsilon_k]\theta_2[\varepsilon_i]=\pi^2\theta[\varepsilon_{lmp}]\Theta_{\{i,k\}}\\
\end{split}
\label{3of6}
\end{align}
Here, $\Theta_{\{i,j\}}=\theta[\varepsilon_{klp}]\theta[\varepsilon_{kmp}]\theta[\varepsilon_{klm}]$, as it is apparent from the construction. $3$-indexed $\varepsilon$ can be changed to $2$-indexed $\varepsilon$ if convenient. Each two of eq. (\ref{3of6}) can be used to solve for $\theta_n[\varepsilon_i]$, $\theta_n[\varepsilon_j]$ or $\theta_n[\varepsilon_k]$, $n=1,2$, and the third one provides a useful substitution. In the course, $\varepsilon_{lmp}$ cancels and we arrive at the following lemma:
\begin{lemma}
For any odd genus-$2$ curve $C$ and one from $20$ triples $\{i,j,k\}\subset\{1,\ldots,6\}$ (with $\varepsilon_6\equiv K_{\infty}$) the following relation holds:
\begin{align}
\theta_n[\varepsilon_i]\Theta_{\{j,k\}}\pm\theta_n[\varepsilon_j]\Theta_{\{i,k\}}=\theta_n[\varepsilon_k]\Theta_{\{i,j\}},
\label{triplerel}
\end{align}
with $n=1,2$ and $\Theta_{\{i,j\}}=\theta[\varepsilon_{klp}]\theta[\varepsilon_{kmp}]\theta[\varepsilon_{klm}]$ and analogously. The characteristics in $\Theta_{\{i,j\}}$ sum up to $\varepsilon_k$.
\label{triplerel2}
\end{lemma}
For the choice above, $\mathcal{I}_0=\{1,2\}$, we deliberately pick as the third index $6$. Eq. (\ref{triplerel}) gives us:
\begin{align}
\begin{split}
\theta_n[\varepsilon_1]\theta[\varepsilon_{135}]\theta[\varepsilon_{145}]\theta[\varepsilon_{134}]-\theta_n[\varepsilon_2]\theta[\varepsilon_{235}]\theta[\varepsilon_{245}]\theta[\varepsilon_{234}]&=\theta_n[K_{\infty}]\theta[\varepsilon_{346}]\theta[\varepsilon_{356}]\theta[\varepsilon_{456}]\\
\equiv\theta_n[\varepsilon_1]\theta[\varepsilon_{24}]\theta[\varepsilon_{23}]\theta[\varepsilon_{25}]-\theta_n[\varepsilon_2]\theta[\varepsilon_{14}]\theta[\varepsilon_{13}]\theta[\varepsilon_{15}]&=\theta_n[K_{\infty}]\theta[\varepsilon_{34}]\theta[\varepsilon_{35}]\theta[\varepsilon_{45}].
\end{split}
\label{126rel}
\end{align}
For other partitions one has to keep in mind the sign in eq. (\ref{triplerel}) and switch the order of the characteristics if required. 

\subsection{General Rosenhain Theorem}
With our chosen partitions eq. (\ref{generalrosenhain}) reads 
\begin{align}\begin{split}
\mathcal{A}^{-1}=\frac{\epsilon}{2\pi^2\Theta_{\{1,2\}}}
\left(\begin{array}{rr}  \theta_2[\varepsilon_1]&  -\theta_2[\varepsilon_2] \\
-\theta_1[\varepsilon_1]&  \theta_1\varepsilon_2]
 \end{array}\right)\mathrm{Diag}\left(\sqrt[4]{ \chi_1}, \sqrt[4]{ \chi_2}  \right).\left( \begin{array}{cc} 
-e_2&1\\
-e_1&1 \end{array} \right).\label{ainverseg2}
\end{split}
\end{align}
Theorem \ref{generalchi} gives us: 
\begin{align}
\begin{split}
\sqrt[4]{\chi_1}=\epsilon\sqrt{e_2-e_1}\frac{\Theta_{2}}{\Theta_{1,2}}\\
\sqrt[4]{\chi_2}=\epsilon\sqrt{e_2-e_1}\frac{\Theta_{1}}{\Theta_{1,2}},
\end{split}
\end{align}
where the previous mentioned reordering of the branch points was applied. Note that $\Theta_1=\Theta_{\{1,6\}}$ and $\Theta_2=\Theta_{\{2,6\}}$.\\
To compare this result with the Rosenhain-memoir \cite{ros851} we apply a Moebius transformation to the curve, which sets $e_1=0$ and $e_2=1$. Now using $\Theta_{\{1,2\}}$, $\Theta_{\{1,6\}}$ and $\Theta_{\{2,6\}}$ as well as eq. (\ref{126rel}) we find:
\begin{align}\begin{split}
\mathcal{A}^{-1}_{1,1}&=-\epsilon\frac{\Theta_{\{2,6\}}}{2\pi^2\Theta_{\{1,2\}}^2} 
\theta_2[\varepsilon_1], \quad 
\mathcal{A}^{-1}_{2,1}=\epsilon\frac{\Theta_{\{2,6\}}}{2\pi^2\Theta_{\{1,2\}}^2} 
\theta_1[\varepsilon_1]\\
\mathcal{A}^{-1}_{1,2}&=\epsilon\frac{1}{2\pi^2\Theta_{\{1,2\}}^2} 
(\Theta_{\{2,6\}}\theta_2[\varepsilon_1]-\Theta_{\{1,6\}}\theta_2[\varepsilon_2])=
\epsilon\frac{1}{2\pi^2\Theta_{\{1,2\}}} \theta_2[K_{\infty}]\\
\mathcal{A}^{-1}_{2,2}&=-\epsilon\frac{1}{2\pi^2\Theta_{\{1,2\}}^2} 
(\Theta_{\{2,6\}}\theta_1[\varepsilon_1]-\Theta_{\{1,6\}}\theta_1[\varepsilon_2])=-
\epsilon\frac{1}{2\pi^2\Theta_{\{1,2\}}} \theta_1[K_{\infty}]
\end{split}
\label{genus2Aminus1}
\end{align}
We now can identify $\delta_1=\varepsilon_1$, $\delta_2=K_{\infty}$, $P=\Theta_{\{2,6\}}$ and $Q=\Theta_{\{1,2\}}$ and hence we have recovered Rosenhain's theorem, eq. (\ref{rosenhain2}), along with the extra identity $(1-a_1)(1-a_2)(1-a_3)=\frac{\Theta^4_{\{1,6\}}}{Q^4}$.\\[0.5cm]
We used here the partition $\{1,2\}\cup\{3,4,5\}$ in order to compare it to Rosenhain's original theorem. But the techniques of Theorem \ref{generalchi} and Lemma \ref{triplerel2} allow for a more general statement:\\
We take the sets $\mathcal{I}_0=\{i,j\}$ and $\mathcal{J}_0=\{k,l,m\}$, all indices mutually disjoint. Again, we normalize the curve to $e_i=0$ and $e_j=1$ by means of a Moebius transformation.\\
One can see, that for a set $\mathcal{I}_0=\{i,j\}$ it is always necessary to pick $6$ as the third index for Lemma \ref{triplerel2} to be applicable in this context. We arrive at the following theorem:
\begin{theorem}[General genus-$2$ Rosenhain Theorem]
For an odd genus-$2$ curve $C$ with normalized branchpoints $e_i=0$, $e_j=1$ and arbitrary branchpoints $e_k$, $e_l$, $e_m$ the inverse period matrix $\mathcal{A}^{-1}$ is given as:
\begin{align}
\mathcal{A}^{-1}=\frac{1}{2\pi^2\Theta^2_{\{i,j\}}}\left[\begin{array}{rr}
-\Theta_{\{j,6\}}\theta_2\left[\varepsilon_i\right]& \Theta_{\{i,j\}}\theta_2\left[K_{\infty}\right]\\
\Theta_{\{j,6\}}\theta_1\left[\varepsilon_i\right]& -\Theta_{\{i,j\}}\theta_1\left[K_{\infty}\right]
\end{array}\right].
\label{finalrosenhain}
\end{align}
\end{theorem}
In the same fashion one can indicate $\mathcal{A}$ if desired. We therefore invert eq. (\ref{finalrosenhain}) using eq. (\ref{3of6}) one time. We conclude:

\begin{align}
\mathcal{A}=\frac{2\Theta_{\{i,j\}}}{\Theta_{\{i,6\}}\Theta_{\{j,6\}}\theta[\varepsilon_{ij}]}\left[\begin{array}{ll}
\Theta_{\{i,j\}}\theta_1[K_{\infty}]&\Theta_{\{i,j\}}\theta_2[K_{\infty}]\\
\Theta_{\{j,6\}}\theta_1[\varepsilon_i]&\Theta_{\{j,6\}}\theta_2[\varepsilon_i]
\end{array}\right].
\end{align}
Note that this formula incorporates all $10$ even characteristics. Also, the three characteristics in $\Theta_{\{i,j\}}$ sum up to (the odd) $K_{\infty}$ and the three characteristics in $\Theta_{\{j,6\}}$ sum up to (the odd) $\varepsilon_i$.

\section{A genus-3 Rosenhain formula}
We take a hyperelliptic\footnote{If not stated otherwise we always mean hyperelliptic curves.} curve in the form, 
\begin{equation}
y^2=\prod_{k=1}^7(x-e_k)\equiv f(x)=\phi(x)\psi(x), \quad e_k\in \mathbb{C}
\end{equation}
where
\begin{equation} \phi(x)=(x-e_1)(x-e_2)(x-e_3), \quad \psi(x)=(x-e_4)(x-e_5)(x-e_6)(x-e_7),   \end{equation}
in especially we fixed $\mathcal{I}_0=\{1,2,3\}$.
The homology basis is the apparent generalization of Fig. \ref{figure-1}.\\
The characteristics of the Abelian images of branch points are
\begin{align*}
[ \boldsymbol{\mathfrak{A}}_1 ]&= \left[ \begin{array}{ccc}  1&0&0\\ 0&0&0  \end{array} \right], 
\quad [ \boldsymbol{\mathfrak{A}}_2 ]= \left[ \begin{array}{ccc}  1&0&0\\ 1&0&0  \end{array} \right],
 \quad[ \boldsymbol{\mathfrak{A}}_3 ]= \left[ \begin{array}{ccc}  0&1&0\\ 1&0&0  \end{array} \right],\quad 
[ \boldsymbol{\mathfrak{A}}_4 ]= \left[ \begin{array}{ccc}  0&1&0\\ 1&1&0  \end{array} \right],  \\
[ \boldsymbol{\mathfrak{A}}_5 ]&= \left[ \begin{array}{ccc}  0&0&1\\ 1&1&0  \end{array} \right], 
\quad [ \boldsymbol{\mathfrak{A}}_6 ]= \left[ \begin{array}{ccc}  0&0&1\\ 1&1&1  \end{array} \right],
 \quad[ \boldsymbol{\mathfrak{A}}_7 ]= \left[ \begin{array}{ccc}  0&0&0\\ 1&1&1  \end{array} \right],\quad 
[ \boldsymbol{\mathfrak{A}}_8 ]= \left[ \begin{array}{ccc}  0&0&0\\ 0&0&0  \end{array} \right]
\end{align*}
The vector of Riemann constants $\boldsymbol{K}_{\infty}$  with base point at $P_8=\infty$ is given in this homology basis as 
\begin{equation}
\boldsymbol{K}_{\infty}= [  \boldsymbol{\mathfrak{A}}_2 ]+[  \boldsymbol{\mathfrak{A}}_4 ]+ [  \boldsymbol{\mathfrak{A}}_6 ]=[  \boldsymbol{\mathfrak{A}}_1 ]+[  \boldsymbol{\mathfrak{A}}_3 ]+ [  \boldsymbol{\mathfrak{A}}_5 ]+[  \boldsymbol{\mathfrak{A}}_7 ]=\left[ \begin{array}{ccc}  1&1&1\\ 1&0&1  \end{array}\right] 
\end{equation}
The important characteristics are here: 
 \begin{align}\begin{split}{}
[ \varepsilon_i] = \left[\boldsymbol{\mathfrak{A}}_i\right] - \left[\boldsymbol{K}_{\infty}  \right], \quad &i=1,\ldots, 8\\
[\varepsilon_{ij}]= \left[\boldsymbol{\mathfrak{A}}_i\right] +\left[\boldsymbol{\mathfrak{A}}_j\right] - \left[\boldsymbol{K}_{\infty}  \right], \quad &i,j=1,\ldots,8, \, i\neq j,\\
[\varepsilon_{ijk}]= \left[\boldsymbol{\mathfrak{A}}_i\right] +\left[\boldsymbol{\mathfrak{A}}_j\right] +\left[\boldsymbol{\mathfrak{A}}_k\right]- \left[\boldsymbol{K}_{\infty}  \right], \quad &i,j,k=1,\ldots,8, \, k\neq i\neq j\neq k.
\end{split}
\end{align}
  
The Riemann-Jacobi formula for this choice of $\mathcal{I}_0$ (and hence $\mathcal{J}_0=\mathcal{T}_0=\{4,5,6,7\}$) reads
\begin{align}
\begin{split}\left.   \mathrm{Det}\; 
\frac{\partial( \theta[\varepsilon_{23}], \theta[\varepsilon_{13}], \theta[\varepsilon_{12}] ) }{\partial (v_1, v_2, v_3)}
\right\vert_{\boldsymbol{v}=0} =\pi^3\;\theta[\varepsilon_{567}]\;\theta[\varepsilon_{467}]\;\theta[\varepsilon_{457}]\;\theta[\varepsilon_{456}]\;\theta[\varepsilon_{4567}]
=\pi^3\Theta_{\{1,2,3\}}\theta[\varepsilon_{4567}].
\end{split}
\label{RJ3}
\end{align}
Following the necessary steps, eq. (\ref{finalanswer1}) gives us for $\mathcal{A}^{-1}=\left(\boldsymbol{U}_1,\boldsymbol{U}_2,\boldsymbol{U}_3\right)$:
\begin{align}
\begin{split}
\boldsymbol{U}_1=\frac{\epsilon}{2\pi^3\Theta_{\{1,2,3\}}}\mathrm{Adj}(J)\left(
\begin{array}{c}
\sqrt[4]{\chi_1}\,e_2e_3\\
\sqrt[4]{\chi_2}\,e_1e_3\\
\sqrt[4]{\chi_3}\,e_1e_2
\end{array}\right)=&\frac{\epsilon}{2\pi^3\Theta_{\{1,2,3\}}}\mathrm{Adj}(J)\left(
\begin{array}{c}
\sqrt[4]{\chi_1}\,e_3\\
0\\
0
\end{array}\right),\\
\boldsymbol{U}_2=\frac{\epsilon}{2\pi^3\Theta_{\{1,2,3\}}}\mathrm{Adj}(J)\left(
\begin{array}{l}
\sqrt[4]{\chi_1}\,\left(e_2-e_3\right)\\
\sqrt[4]{\chi_2}\,\left(e_1-e_3\right)\\
\sqrt[4]{\chi_3}\,\left(e_1-e_2\right)
\end{array}\right)=&\frac{\epsilon}{2\pi^3\Theta_{\{1,2,3\}}}\mathrm{Adj}(J)\left(
\begin{array}{l}
-\sqrt[4]{\chi_1}\,\left(e_3-1\right)\\
-\sqrt[4]{\chi_2}\,e_3\\
-\sqrt[4]{\chi_3}
\end{array}\right),\\
\boldsymbol{U}_3=&\frac{\epsilon}{2\pi^3\Theta_{\{1,2,3\}}}\mathrm{Adj}(J)\left(
\begin{array}{c}
\sqrt[4]{\chi_1}\\
\sqrt[4]{\chi_2}\\
\sqrt[4]{\chi_3}
\end{array}\right),
\end{split}
\label{U1U2U3}
\end{align}
where we normalized again to $e_1=0,\, e_2=1$ ,but $e_3$ can't be expressed within our technique in the resulting formulae. We now can insert $\chi_k$ from Lemma \ref{generalchi} into eq. (\ref{U1U2U3}):
\begin{align}
\begin{split}
\boldsymbol{U}_1&=\frac{\epsilon}{2\pi^3\Theta_{\{1,2,3\}}^{3/2}}\mathrm{Adj}(J)\left(
\begin{array}{l}
\Theta_{\{23\}}^{1/2}\cdot e_3^{5/4}\\
0\\
0
\end{array}\right),\\
\boldsymbol{U}_2&=\frac{\epsilon}{2\pi^3\Theta_{\{1,2,3\}}^{3/2}}\mathrm{Adj}(J)\left(
\begin{array}{l}
-\Theta_{\{23\}}^{1/2}\cdot e_3^{1/4}\cdot\left(e_3-1\right)\\
-\Theta_{\{13\}}^{1/2}\cdot e_3\cdot\left(e_3-1\right)^{1/4}\\
-\Theta_{\{12\}}^{1/2}\cdot e_3^{1/4}\cdot\left(e_3-1\right)^{1/4}
\end{array}\right),\\
\boldsymbol{U}_3&=\frac{\epsilon}{2\pi^3\Theta_{\{1,2,3\}}^{3/2}}\mathrm{Adj}(J)\left(
\begin{array}{l}
\Theta_{\{23\}}^{1/2}\cdot e_3^{1/4}\\
\Theta_{\{13\}}^{1/2}\cdot\left(e_3-1\right)^{1/4}\\
\Theta_{\{12\}}^{1/2}\cdot e_3^{1/4}\cdot\left(e_3-1\right)^{1/4}
\end{array}\right).
\end{split}
\label{U1U2U3final}
\end{align}
If required, we could use eq. (\ref{notinverseA}) to arrive at $\mathcal{A}$. But currently we see no further simplifications and therefore didn't depict it here.
\section{Concluding remarks \& Acknowledgments}
Without any major changes, one can adopt the method shown for genus $3$ to higher genera. Unfortunately we were not able to find a generalization to Lemma \ref{triplerel2}, which could bring eq. (\ref{U1U2U3final}) down to a structure like in eq. (\ref{rosenhain2}). It seems unlikely that there exists one as simple as in genus $2$.\\
Our next steps in this work could be to unfix the base point, which was infinity throughout this work. And we see a chance to develop Thomae type formulae expressing higher derivative $\theta$-constants. We hope to come back to this topic in the near future.\\[0.5cm]
The author wants to thank Victor Enolskii for providing the idea of the work and many suggestions for useful techniques as well as the constant interest in the work.
Also the author gratefully acknowledges the Deutsche Forschungsgemeinschaft (DFG) for financial support within the framework of the DFG Research Training group 1620 Models of gravity.

\section{Appendix A: Rosenhain derivative formulae}

For any two odd characteristics $[\delta_1]$, $[\delta_2]$ denote 
\[  D \left[  \delta_1;\delta_2  \right] = \theta_1[\delta_1] \theta_2[\delta_2]- \theta_2[\delta_1] \theta_1[\delta_2]  \]
Then the following 15 Rosenhain derivative formulae are valid
\begin{eqnarray*}
D\left(
\left[\begin{array}{ll}\scriptstyle{0}&\!\!\!\!\scriptstyle{1}\cr
\scriptstyle{0}&\!\!\!\!\scriptstyle{1}\end{array}\right],
\left[\begin{array}{ll}\scriptstyle{1}&\!\!\!\!\scriptstyle{1}\cr
\scriptstyle{0}&\!\!\!\!\scriptstyle{1}\end{array}\right]
\right)=\pi^2
\theta\ma{0}{0}{1}{0}\theta\ma{0}{0}{1}{1}\theta\ma{1}{1}{1}{1}
\theta\ma{0}{1}{1}{0},\quad\left(\ma{1}{0}{0}{0}\right);
\end{eqnarray*}
\begin{eqnarray*}
D\left(
\left[\begin{array}{ll}\scriptstyle{1}&\!\!\!\!\scriptstyle{1}\cr
\scriptstyle{1}&\!\!\!\!\scriptstyle{0}\end{array}\right],
\left[\begin{array}{ll}\scriptstyle{1}&\!\!\!\!\scriptstyle{0}\cr
\scriptstyle{1}&\!\!\!\!\scriptstyle{0}\end{array}\right]
\right)=\pi^2
\theta\ma{1}{1}{1}{1}\theta\ma{0}{0}{0}{1}\theta\ma{0}{0}{1}{1}
\theta\ma{1}{0}{0}{1},\quad\left(\ma{0}{1}{0}{0}\right);
\end{eqnarray*}
\begin{eqnarray*}
D\left(
\left[\begin{array}{ll}\scriptstyle{1}&\!\!\!\!\scriptstyle{0}\cr
\scriptstyle{1}&\!\!\!\!\scriptstyle{1}\end{array}\right],
\left[\begin{array}{ll}\scriptstyle{0}&\!\!\!\!\scriptstyle{1}\cr
\scriptstyle{1}&\!\!\!\!\scriptstyle{1}\end{array}\right]
\right)=\pi^2
\theta\ma{0}{1}{1}{0}\theta\ma{1}{0}{0}{1}\theta\ma{0}{0}{1}{0}
\theta\ma{0}{0}{0}{1},\quad\left(\ma{1}{1}{0}{0}\right);
\end{eqnarray*}
\begin{eqnarray*}
D\left(
\left[\begin{array}{ll}\scriptstyle{0}&\!\!\!\!\scriptstyle{1}\cr
\scriptstyle{1}&\!\!\!\!\scriptstyle{1}\end{array}\right],
\left[\begin{array}{ll}\scriptstyle{0}&\!\!\!\!\scriptstyle{1}\cr
\scriptstyle{0}&\!\!\!\!\scriptstyle{1}\end{array}\right]
\right)=\pi^2
\theta\ma{1}{0}{0}{0}\theta\ma{1}{0}{0}{1}\theta\ma{1}{1}{0}{0}
\theta\ma{1}{1}{1}{1},\quad\left(\ma{0}{0}{1}{0}\right);
\end{eqnarray*}

\begin{eqnarray*}
D\left(
\left[\begin{array}{ll}\scriptstyle{0}&\!\!\!\!\scriptstyle{1}\cr
\scriptstyle{1}&\!\!\!\!\scriptstyle{1}\end{array}\right],
\left[\begin{array}{ll}\scriptstyle{1}&\!\!\!\!\scriptstyle{1}\cr
\scriptstyle{0}&\!\!\!\!\scriptstyle{1}\end{array}\right]
\right)=\pi^2
\theta\ma{0}{0}{0}{0}\theta\ma{0}{0}{0}{1}\theta\ma{1}{1}{1}{1}
\theta\ma{0}{1}{0}{0},\quad\left(\ma{1}{0}{1}{0}\right);
\end{eqnarray*}
\begin{eqnarray*}
D\left(
\left[\begin{array}{ll}\scriptstyle{1}&\!\!\!\!\scriptstyle{0}\cr
\scriptstyle{1}&\!\!\!\!\scriptstyle{1}\end{array}\right],
\left[\begin{array}{ll}\scriptstyle{1}&\!\!\!\!\scriptstyle{1}\cr
\scriptstyle{0}&\!\!\!\!\scriptstyle{1}\end{array}\right]
\right)=\pi^2
\theta\ma{1}{1}{0}{0}\theta\ma{0}{0}{1}{1}\theta\ma{0}{0}{0}{1}
\theta\ma{1}{0}{0}{0},\quad\left(\ma{0}{1}{1}{0}\right);
\end{eqnarray*}
\begin{eqnarray*}
D\left(
\left[\begin{array}{ll}\scriptstyle{1}&\!\!\!\!\scriptstyle{0}\cr
\scriptstyle{1}&\!\!\!\!\scriptstyle{1}\end{array}\right],
\left[\begin{array}{ll}\scriptstyle{0}&\!\!\!\!\scriptstyle{1}\cr
\scriptstyle{0}&\!\!\!\!\scriptstyle{1}\end{array}\right]
\right)=\pi^2
\theta\ma{0}{1}{0}{0}\theta\ma{1}{0}{0}{1}\theta\ma{0}{0}{0}{0}
\theta\ma{0}{0}{1}{1},\quad\left(\ma{1}{1}{1}{0}\right);\end{eqnarray*}

\begin{eqnarray*}
D\left(
\left[\begin{array}{ll}\scriptstyle{1}&\!\!\!\!\scriptstyle{0}\cr
\scriptstyle{1}&\!\!\!\!\scriptstyle{0}\end{array}\right],
\left[\begin{array}{ll}\scriptstyle{1}&\!\!\!\!\scriptstyle{0}\cr
\scriptstyle{1}&\!\!\!\!\scriptstyle{1}\end{array}\right]
\right)=\pi^2
\theta\ma{0}{1}{0}{0}\theta\ma{0}{1}{1}{0}\theta\ma{1}{1}{1}{1}
\theta\ma{1}{1}{0}{0},\quad\left(\ma{0}{0}{0}{1}\right);
\end{eqnarray*}
\begin{eqnarray*}
D\left(
\left[\begin{array}{ll}\scriptstyle{1}&\!\!\!\!\scriptstyle{1}\cr
\scriptstyle{1}&\!\!\!\!\scriptstyle{0}\end{array}\right],
\left[\begin{array}{ll}\scriptstyle{0}&\!\!\!\!\scriptstyle{1}\cr
\scriptstyle{1}&\!\!\!\!\scriptstyle{1}\end{array}\right]
\right)=\pi^2
\theta\ma{0}{0}{1}{1}\theta\ma{0}{0}{1}{0}\theta\ma{1}{1}{0}{0}
\theta\ma{0}{1}{0}{0},\quad\left(\ma{1}{0}{0}{1}\right);\end{eqnarray*}
\begin{eqnarray*}
D\left(
\left[\begin{array}{ll}\scriptstyle{1}&\!\!\!\!\scriptstyle{1}\cr
\scriptstyle{1}&\!\!\!\!\scriptstyle{0}\end{array}\right],
\left[\begin{array}{ll}\scriptstyle{1}&\!\!\!\!\scriptstyle{0}\cr
\scriptstyle{1}&\!\!\!\!\scriptstyle{1}\end{array}\right]
\right)=\pi^2
\theta\ma{1}{1}{1}{1}\theta\ma{0}{0}{0}{0}\theta\ma{0}{0}{1}{0}
\theta\ma{1}{0}{0}{0},\quad\left(\ma{0}{1}{0}{1}\right);\end{eqnarray*}
\begin{eqnarray*}
D\left(
\left[\begin{array}{ll}\scriptstyle{1}&\!\!\!\!\scriptstyle{0}\cr
\scriptstyle{1}&\!\!\!\!\scriptstyle{0}\end{array}\right],
\left[\begin{array}{ll}\scriptstyle{0}&\!\!\!\!\scriptstyle{1}\cr
\scriptstyle{1}&\!\!\!\!\scriptstyle{1}\end{array}\right]
\right)=\pi^2
\theta\ma{0}{1}{1}{0}\theta\ma{1}{0}{0}{0}\theta\ma{0}{0}{1}{1}
\theta\ma{0}{0}{0}{0},\quad\left(\ma{1}{1}{0}{1}\right);\end{eqnarray*}

\begin{eqnarray*}
D\left(
\left[\begin{array}{ll}\scriptstyle{1}&\!\!\!\!\scriptstyle{1}\cr
\scriptstyle{1}&\!\!\!\!\scriptstyle{0}\end{array}\right],
\left[\begin{array}{ll}\scriptstyle{1}&\!\!\!\!\scriptstyle{1}\cr
\scriptstyle{0}&\!\!\!\!\scriptstyle{1}\end{array}\right]
\right)=\pi^2
\theta\ma{1}{0}{0}{1}\theta\ma{1}{0}{0}{0}\theta\ma{0}{1}{1}{0}
\theta\ma{0}{1}{0}{0},\quad\left(\ma{0}{0}{1}{1}\right);\end{eqnarray*}
\begin{eqnarray*}
D\left(
\left[\begin{array}{ll}\scriptstyle{1}&\!\!\!\!\scriptstyle{1}\cr
\scriptstyle{1}&\!\!\!\!\scriptstyle{0}\end{array}\right],
\left[\begin{array}{ll}\scriptstyle{0}&\!\!\!\!\scriptstyle{1}\cr
\scriptstyle{0}&\!\!\!\!\scriptstyle{1}\end{array}\right]
\right)=\pi^2
\theta\ma{0}{0}{0}{1}\theta\ma{0}{0}{0}{0}\theta\ma{1}{1}{0}{0}
\theta\ma{0}{1}{1}{0},\quad\left(\ma{1}{0}{1}{1}\right);\end{eqnarray*}
\begin{eqnarray*}
D\left(
\left[\begin{array}{ll}\scriptstyle{1}&\!\!\!\!\scriptstyle{0}\cr
\scriptstyle{1}&\!\!\!\!\scriptstyle{0}\end{array}\right],
\left[\begin{array}{ll}\scriptstyle{1}&\!\!\!\!\scriptstyle{1}\cr
\scriptstyle{0}&\!\!\!\!\scriptstyle{1}\end{array}\right]
\right)=\pi^2
\theta\ma{1}{1}{0}{0}\theta\ma{0}{0}{1}{0}\theta\ma{0}{0}{0}{0}
\theta\ma{1}{0}{0}{1},\quad\left(\ma{0}{1}{1}{1}\right);\end{eqnarray*}
\begin{eqnarray*}
D\left(
\left[\begin{array}{ll}\scriptstyle{1}&\!\!\!\!\scriptstyle{0}\cr
\scriptstyle{1}&\!\!\!\!\scriptstyle{0}\end{array}\right],
\left[\begin{array}{ll}\scriptstyle{0}&\!\!\!\!\scriptstyle{1}\cr
\scriptstyle{0}&\!\!\!\!\scriptstyle{1}\end{array}\right]
\right)=\pi^2
\theta\ma{0}{1}{0}{0}\theta\ma{1}{0}{0}{0}\theta\ma{0}{0}{0}{1}
\theta\ma{0}{0}{1}{0},\quad\left(\ma{1}{1}{1}{1}\right).
\end{eqnarray*}
We pointed at the right margin the characteristic, which is the
sum of characteristics of each entry to the corresponding
equality.

\end{document}